\documentclass[11pt,twoside]{amsart}
\usepackage[latin1]{inputenc}
\usepackage[OT1]{fontenc}
\usepackage{amsmath}
\usepackage{amsthm}
\usepackage{amssymb}
\usepackage[all]{xy}
\usepackage{bbm}
\usepackage{amscd}
\usepackage{a4wide}
\usepackage{enumitem}
\usepackage{mathtools}

\setenumerate[0]{label=(\roman*)}

\DeclareMathOperator{\cc}{\mathsf{c}}

\DeclareMathOperator{\pr}{\mathsf{pr}}

\DeclareMathOperator{\Ho}{\mathsf H}

\DeclareMathOperator{\rank}{\mathsf{rank}}
\DeclareMathOperator{\ch}{\mathsf{ch}}

\DeclareMathOperator{\CH}{\mathsf{CH}}

\let\det\relax
\DeclareMathOperator{\det}{\mathsf{det}}

\newcommand{\wH}{\widetilde H}

\makeatletter
\newcommand{\leqnomode}{\tagsleft@true}
\newcommand{\reqnomode}{\tagsleft@false}
\makeatother

\newcommand{\sym}{\mathfrak S}

\newcommand{\cM}{\mathcal M}

\newcommand{\reg}{\mathcal O}

\renewcommand{\theta}{\vartheta}
\renewcommand{\rho}{\varrho}
\renewcommand{\phi}{\varphi}

\usepackage[usenames,dvipsnames]{xcolor}
\usepackage{hyperref}
\hypersetup{
    colorlinks,
    linkcolor={red!50!black},
    citecolor={blue!50!black},
    urlcolor={blue!80!black}
}
\usepackage{aliascnt} 

\newtheorem{theorem}{Theorem}[section]

  \newaliascnt{proposition}{theorem}
  
  \aliascntresetthe{proposition}

  \newaliascnt{lemma}{theorem}
  \newtheorem{lemma}[lemma]{Lemma}
  \aliascntresetthe{lemma}

  \newaliascnt{corollary}{theorem}
  \newtheorem{cor}[corollary]{Corollary}
  \aliascntresetthe{corollary}

\theoremstyle{definition}

  \newaliascnt{definition}{theorem}
  
  \aliascntresetthe{definition}

  \newaliascnt{remark}{theorem}
  
  \aliascntresetthe{remark}

  \newaliascnt{condition}{theorem}
  
  \aliascntresetthe{condition}

  \newaliascnt{question}{theorem}
  
  \aliascntresetthe{question}

  \newaliascnt{example}{theorem}
  
  \aliascntresetthe{example}

\begin{document}

\title{Discriminant of Tautological Bundles on Symmetric Products of Curves}
\author[A.\ Krug]{Andreas Krug}

\begin{abstract}
We compute a formula for the discriminant of tautological bundles on symmetric powers of a complex smooth projective curve. 
It follows that the Bogomolov inequality does not give a new restriction to stability of these tautological bundles. It only rules out tautological bundles which are already known to have the structure sheaf as a destabilising subbundle.
\end{abstract}

\maketitle

\section{Introduction}

There is a natural way to produce a vector bundle on the symmetric product of a curve out of a vector bundle on the curve -- use the Fourier--Mukai transform along the universal effective divisor. 

More precisely, let $C$ be a smooth curve over an algebraically closed field of characteristic zero, and let $n\ge 2$. Then the symmetric product $C^{(n)}=C^n/\sym_n$ can be identified with the moduli space of effective degree $n$ divisors on $C$, with the universal family $\Xi\subset C\times C^{(n)}$ given by the image of the embedding 
\[C\times C^{(n-1)}\hookrightarrow C\times C^{(n)}\quad,\quad (x,x_1+\dots +x_{n-1})\mapsto (x,x_1+\dots+x_{n-1}+x)\,.\]
Let $p\colon \Xi\to C$ and $q\colon \Xi\to C^{(n)}$ be the projections. For a vector bundle $E$ on $C$, the associated \textit{tautological bundle} on $C^{(n)}$ is $E^{[n]}:=q_*p^*E$. Since $q$ is flat and finite of degree $n$, this is indeed a vector bundle with $\rank(E^{[n]})=n\rank(E)$ and fibres $E^{[n]}(x_1+\dots+x_n)=\Ho^0(E_{\mid x_1+\dots+x_n})$. 

One much-studied question regarding tautological bundles is whether they are stable. Here, by \emph{stability} we mean slope stability with respect to the polarisation $H=x+C^{(n-1)}\subset C^{(n)}$.
Strengthening earlier results of \cite{AnconaOttaviani--stab, BohnhorstSpindler--stab, Mistretta--stabtaut, EMLN-stab, BiswasNagaraj-stab, DanPal-stab, BasuDan-stab}, it was proven in \cite{Krug--stab} that if $E$ is (semi-)stable with slope $\mu(E)\notin[-1,n-1]$ (or, for semi-stability, $\mu(E)\notin(-1,n-1)$), then also $E^{[n]}$ is (semi-)stable.

For $\mu(E)\in [-1,n-1]$, the situation seems more complicated.
It was observed in \cite[Sect.\ 3]{Krug--stab} that the above condition for stability is already numerically optimal: For every curve $C$ and every $d=0,\dots, n-2$ there is a line bundle of degree $d$ on $C$ such that $L^{[n]}$ is unstable. There are also line bundles of degree $-1$ and $n-1$ such that $L^{[n]}$ is strictly semi-stable. On the other hand, there are also examples of line bundles with $d\in(-1,n-1)$ such that $L^{[n]}$ is stable. Namely, by \cite{BiswasNagaraj-stab}, $L^{[2]}$ is stable for every non-trivial line bundle $L$ of degree zero. We would like to find a complete answer to the following question:

\textit{Let $E$ be a (semi-)stable bundle of slope $\mu(E)\in [-1,n-1]$. Under which circumstances is $E^{[n]}$ a (semi-)stable bundle on $C^{(n)}$?}

An obvious first approach to get a meaningful restriction on semi-stability of $E^{[n]}$ is to use the Bogomolov inequality
\[
\int_{C^{(n)}} \Delta(E^{[n]})H^{n-2}\ge 0  
\]
where $\Delta(E^{[n]})=-2\rank(E^{[n]})\ch_2(E^{[n]})+\cc_1(E^{[n]})^2$ is the discriminant; see \cite{Bogomolov--ineq}, \cite{Miyaoka--stabineq}, \cite[Thm.\ 7.3.1]{HL}. In this paper, we compute the relevant intersection number. 

\begin{theorem}\label{thm:main}
Let $C$ be a smooth projective curve of genus $g$, and let $E$ be a vector bundle on $C$ of rank $r$ and degree $d$. Then, for every $n\ge 2$, 
\[
 \int_{C^{(n)}} \Delta(E^{[n]})H^{n-2}=d^2-(n-2)dr+(n-1)(g-1)r^2\,.
\]
\end{theorem}
Dividing the equation of \autoref{thm:main} by $r^2$, we see that 
\[
\int \Delta(E^{[n]})H^{n-2}\ge 0 \quad\iff\quad \mu^2-(n-2)\mu+(n-1)(g-1)\ge 0\,. 
\]
Hence, the Bogomolov inequality gives

\begin{cor}\label{cor:unstable}
The tautological bundle $E^{[n]}$ is unstable for every vector bundle $E$ on $C$ with
\[
 \mu(E)\in \left(\frac{n-2}2 - \frac12\sqrt{(n-2)^2-4(n-1)(g-1)}\,\,,\,\, \frac{n-2}2 + \frac12\sqrt{(n-2)^2-4(n-1)(g-1)}      \right)
\]
\end{cor}

For $g\ge 1$, the interval of \autoref{cor:unstable} lies inside the interval $(g-1,n-g)$, and it approaches for large $n$ asymptotically the whole interval $(g-1, n-g)$. 

However, the unstability criterion of \autoref{cor:unstable} is weaker than one already given in \cite[Sect.\ 3]{Krug--stab}. There, it was observed that $\reg_{C^{(n)}}$ is a destabilising subbundle of $E^{[n]}$ for every $E$ with $\mu(E)<n-1$ and $\Ho^0(E)\neq 0$. By the Riemann-Roch theorem, this implies that $E^{[n]}$ is always unstable when $\mu(E)\in (g-1,n-1)$.  

Let now $n=2$, such that $C^{(2)}$ is a surface. Then for $d\notin [-1,r]$, the bundle $E^{[2]}$ is stable by \cite{Krug--stab}, and we write $\cM=\cM_{C^{(2)}}\bigl(\ch(E^{[2]}), \det(E^{[2]})\bigr)$ for the moduli space of stable sheaves on $C^{(2)}$ with the same Chern character and determinant. \autoref{thm:main} tells us that $\int \Delta(E^{[n]})$ grows quadratically with $d$. Hence, we can apply the results of \cite{OG--basic} to obtain
\begin{cor}
 For $|d|\gg 0$, the moduli space $\cM$ is irreducible and generically reduced of the expected dimension.
\end{cor}
Since $\chi(\reg_{C^{(2)}})=1-g+\binom g2=\frac12(g^2-3g+2)$, the expected dimension of $\cM$ is
\[
 \Delta(E^{[2]})-(\rank(E^{[2]})^2-1)\chi(\reg_X)=d^2+r^2(g-1)-\frac12(4r^2-1)(g^2-3g+2)\,.
\]

\section{Proof of \autoref{thm:main}}

\subsection{General Conventions}

All our computations are carried out in the Chow rings modulo numerical equivalence. In particular, we will often omit the integral symbol $\int$ when writing intersection numbers. 

Throughout $C$ will denote a smooth projective curve over a field of characteristic zero, and $x$ will denote a point of the curve $C$. It will not matter which point, since we are doing our computations modulo numerical equivalence. In particular, the notion of slope stability with respect to our polarisation $H_n=x+C^{(n-1)}$ of $C^{(n)}$ does not depend on the choice of $x\in C$.

\subsection{Some intersection numbers on $C^n$}

Let $\pi_n\colon C^n\to C^{(n)}$ be the $\sym_n$-quotient morphism. We have $\wH_n:=\pi_n^*H_n=\sum_{i=1}^n \pr_i^*([x])$ where $\pr_i\colon C^n\to C$ is the projection to the $i$-th factor. 
For $1\le i<j\le n$, we consider the \textit{pairwise diagonal} 
\[
\Delta_{ij}=\bigl\{(x_1,\dots, x_n)\in C^n\mid x_i=x_j\bigr\}\,.
\]
The \textit{big diagonal} is the sum of all pairwise diagonals
\[
 \delta_n=\sum_{1\le i< j\le n}\Delta_{ij}\,.
\]
Another divisor on $C^n$ that we will need is
$\delta_n'=\delta_n-\overline\pr_1^*\delta_n=\sum_{j=2}^n \Delta_{1j}$ where
\[
 \overline \pr_1\colon C^n\to C^{n-1}\quad,\quad (x_1,x_2,\dots,x_n)\mapsto (x_2,\dots,x_n)\,.  
\]
We will often omit the index $n$ in the notation and write $\pi$, $H$, $\wH$, $\delta$, $\delta'$ instead of  
$\pi_n$, $H_n$, $\wH_n$, $\delta_n$, $\delta'_n$.
For $I\subset \{1,\dots,n\}$, we define
\[
 \eta_I:=\prod_{i\in I} \pr_i^*([x])=\bigl[\{(x_1,\dots,x_n)\mid x_i=x\,\,\forall i\in I\}\bigr]\,.
\]
Since $\pr_i^*([x])^2=0$ for every $i=1,\dots, n$, we get 
\begin{equation}\label{eq:wHk}
 \wH^k=k!\sum_{|I|=k} \eta_I\quad\quad\text{for all $k=1,\dots, n$}\,.
\end{equation}

\begin{lemma}
We have the following intersection numbers on $C^n$:
\begin{align}
\wH^n=n!\quad,\quad \delta \wH^{n-1}=n!(n-1) \quad,\quad \delta^2\wH^{n-2}=-n!(g-1)+n!(n-2)n\,,\label{eq:intnumbers1}\\
\pr_1^*([x])\delta'\wH^{n-2}=(n-1)!\quad,\quad \delta'^2\wH^{n-2}=-(n-1)!2(g-1)+(n-1)!(n-2)3\,. \label{eq:intnumbers2}
\end{align}
\end{lemma}
\begin{proof}
The first equation is just a special case of \eqref{eq:wHk}. 
For the second equation, note that, for $|I|=n-1$, we have
\[
\Delta_{ij}\eta_I=\begin{cases} 
                                                                        1\quad& \text{if $|I\cap \{i,j\}|=1$}\,,\\
                                                                        0\quad& \text{if $|I\cap \{i,j\}|=2$.}                                                                    
                                                                       \end{cases}
 \]
Hence, by \eqref{eq:wHk}, we get $\Delta_{ij}\wH^{n-1}=(n-1)!2$. Since there are $\binom n2$ pairwise diagonals, we get $\delta\wH^{n-1}=\binom n2(n-1)!2=n!(n-1)$.

We have $\pr_1^*([x])\Delta_{1j}=\eta_{\{1,j\}}$. Hence, the only summand of $\wH^{n-2}$ which has a non-zero intersection pairing with $\pr_1^*([x])\Delta_{1j}$ is $\eta_{\{2,\dots,n\}\setminus\{j\}}$.  Hence $\pr_1^*([x])\Delta_{1j}\wH^{n-2}=(n-2)!$, again by \eqref{eq:wHk}. The formula 
$\pr_1^*([x])\delta'\wH^{n-2}=(n-2)!(n-1)=(n-1)!$ follows as $\delta'$ consists of $(n-1)$ pairwise diagonals.
For the computation of $\delta'^2\wH^{n-2}$, we first note that
\begin{equation}\label{eq:d1}
\delta'^2=\sum_{j=2}^n \Delta_{1j}^2 +2\sum_{2\le i<j\le n}\Delta_{1i}\Delta_{1j}\,.
\end{equation}
The self-intersection of a pairwise diagonal is $\Delta_{1j}^2=-2(g-1) \eta_{\{1,j\}}$. Hence,
\begin{equation}\label{eq:d2}
\Delta_{1j}^2\wH^{n-2}=-(n-2)!2(g-1)\,.
\end{equation}
Furthermore, $\Delta_{1i}\Delta_{1j}=\Delta_{1ij}=\bigl\{(x_1,\dots,x_n)\mid x_1=x_i=x_j\bigr\}$.   
For $|I|=n-2$, we have
\[
\Delta_{1ij}\eta_I=\begin{cases} 
                                                                        1\quad& \text{if $|I\cap \{1,i,j\}|=1$}\,,\\
                                                                        0\quad& \text{if $|I\cap \{1,i,j\}|\ge2$.}                                                                    
                                                                       \end{cases}
 \]
Since there are three $I\subset \{1,\dots, n\}$ with $|I|=n-2$ and $|I\cap \{1,i,j\}|=1$, equation \eqref{eq:wHk} yields 
\begin{equation}\label{eq:d3}
 \Delta_{1ij}\wH^{n-2}=(n-2)!3\,.
\end{equation}
We obtain the formula for $\delta'^2\wH^{n-2}$ by combining \eqref{eq:d1}, \eqref{eq:d2}, and \eqref{eq:d3}. 

The computation of $\delta^2\wH^{n-2}$ is very similar using that 
\[
 \delta^2=\sum_{1\le i<j\le n}\Delta_{ij}^2+6\sum_{1\le i<j<k\le n}\Delta_{ijk}+\sum_{\{i,j\}\cap\{k,\ell\}=\emptyset}\Delta_{ij}\Delta_{k,\ell}
\]
instead of \eqref{eq:d1}. The only additional ingredient is that, for $\{i,j\}\cap \{k,\ell\}=\emptyset$ and $|I|=n-2$, we have
\[
\Delta_{ij}\Delta_{k\ell}\eta_I=\begin{cases} 
                                                                        1\quad& \text{if $|I\cap \{i,j\}|=1=|I\cap \{k,\ell\}|$}\,,\\
                                                                        0\quad& \text{else.}                                                                    
                                                                       \end{cases}\qedhere
 \]
\end{proof}

\begin{lemma}\label{lem:projformula}
 Let $\alpha\in \CH^2(C^{n-2})$. Then 
 \[
\int_{C^n} \wH_n^{n-2} \overline \pr_1^*\alpha=(n-2) \int_{C^{n-1}} \wH_{n-1}^{n-3} \alpha\,.   
 \]
\end{lemma}

\begin{proof}
 By projection formula, 
 \[
  \int_{C^n} \wH_n^{n-2} \overline \pr_1^*\alpha= \int_{C^{n-1}} \overline\pr_{1*}(\wH_n^{n-2} \overline \pr_1^*\alpha) = \int_{C^{n-1}} (\overline\pr_{1*}\wH_n^{n-2}) \alpha\,.
 \]
We have $\overline\pr_{1*}\wH_n^{n-2}=(n-2)\wH_{n-1}^{n-3}$ as follows form \eqref{eq:wHk} together with 
\[
\overline \pr_{1*}\eta_I=\begin{cases} 
                                                                        \eta_{I\setminus\{1\}}\quad& \text{if $1\in I$,}\\
                                                                        0\quad& \text{if $1\notin I$.}                                                                    
                                                                       \end{cases}\qedhere
 \]
\end{proof}

\subsection{Computation of the discriminant by induction}

The Chern characters of tautological bundles are computed in \cite[Sect.\ 3]{Mattuck--sym} in the case that $E=L$ is a line bundle. So one way to proceed would be to generalise the formula \cite[Prop.\ 6]{Mattuck--sym} to higher rank bundles, and then use it to compute $\Delta(E^{[n]})H^{n-2}$. Instead, we chose a more direct route to the computation of $\Delta(E^{[n]})H^{n-2}$, using the short exact sequence 
\begin{equation}\label{eq:keyses}
0\to \pr_1^*E(-\delta'_n)\to \pi_n^* E^{[n]}\to \overline\pr_1^*\pi_{n-1}^*E^{[n-1]}\to 0\,. 
\end{equation}
of \cite[Prop.\ 1.4]{Krug--stab} to proceed by induction on $n$.

We consider a vector bundle $E$ of rank $r$ and degree $d$ on the curve $C$ of genus $g$.
\begin{lemma}\label{lem:c_1^2} For all $n\ge 2$, we have
 \[
 \cc_1(E^{[n]})^2H^{n-2}=d^2-2dr(n-1)-r^2(g-1)+r^2n(n-2)\,.
 \]
\end{lemma}

\begin{proof}
Using the short exact sequence \eqref{eq:keyses} one deduces by induction the formula
\[
 \cc_1(\pi^*E^{[n]})=d\wH- r\delta\,;
\]
see \cite[Sect.\ 1.7]{Krug--stab}. Hence, $\cc_1(\pi^*E^{[n]})^2=d^2\wH^2-2dr\delta\wH+r^2\delta^2$. By \eqref{eq:intnumbers1}, we get 
\[
 \cc_1(\pi^*E^{[n]})^2\wH^{n-2}=n!\bigl(d^2-2dr(n-1)-r^2(g-1)+r^2n(n-2)\bigr)\,.
\]
Since $\pi\colon C^n\to C^{(n)}$ is of degree $n!$, we have $\cc_1(E^{[n]})^2H^{n-2}=\frac{\cc_1(\pi^*E^{[n]})^2\wH^{n-2}}{n!}$. 
\end{proof}

\begin{lemma}\label{lem:ch_2}
For every $n\ge 2$, we have
\[
 \ch_2(E^{[n]})H^{n-2}=-\frac12\bigl(d+(g-1)r-(n-2)r\bigr)\,.
 \]
\end{lemma}

\begin{proof}
We have $\ch(\pr_1^*E)=r+d\pr_1^*([x])$ and $\ch(\reg(-\delta'))=1-\delta'+\frac12 \delta'^2 \pm \dots$. Hence,
$\ch_2(\pr_1^*E(-\delta'))=-d\pr_1^*([x])\delta'+\frac12 r \delta'^2$ which by \eqref{eq:intnumbers2} gives 
\begin{equation}\label{eq:ch2pr1}
\ch_2(\pr_1^*E(-\delta'))\wH^{n-2}= -(n-1)!d-(n-1)!(g-1)r+(n-1)!(n-2)\frac32r 
\end{equation}
We now can prove the assertion, or rather the equivalent formula
 \begin{equation}\label{eq:ch2pi}
 \ch_2(\pi_n^*E^{[n]})\wH^{n-2}=-\frac12n!\bigl(d+(g-1)r-(n-2)r\bigr)\,,
 \end{equation}
by induction using the short exact sequence \eqref{eq:keyses}. 
For $n=2$, the sequence gives $\ch_2(\pi_2^*E^{[2]})=\ch_2(\pr_1^*E(-\delta'))$, and we check that the right-hand sides of
\eqref{eq:ch2pr1} and \eqref{eq:ch2pi} are equal for $n=2$. For $n\ge 3$, the sequence \eqref{eq:keyses} together with \autoref{lem:projformula} give
\begin{align*}
\ch_2(\pi_n^*E^{[n]})\wH_n^{n-2}= \ch_2(\pr_1^*E(-\delta'))\wH_n^{n-2}+\overline \pr_1^*\ch_2(\pi_{n-1}^*E^{[n-1]})\wH_n^{n-2}\\
= \ch_2(\pr_1^*E(-\delta'))\wH_n^{n-2}+(n-2)\ch_2(\pi_{n-1}^*E^{[n-1]})\wH_{n-1}^{n-3}\,.
\end{align*}
Now, the induction step is done by plugging \eqref{eq:ch2pr1} and the $n-1$ case of \eqref{eq:ch2pi} into the right-hand side.
\end{proof}
Recall that the discriminant is defined as
\[
 \Delta(E^{[n]})=-2\rank(E^{[n]})\ch_2(E^{[n]})+\cc_1(E^{[n]})^2=-2nr\ch_2(E^{[n]})+\cc_1(E^{[n]})^2\,.
\]
Hence \autoref{lem:c_1^2} and \autoref{lem:ch_2} together compute $\Delta(E^{[n]})H^{n-2}$ and prove \autoref{thm:main}.

\bibliographystyle{alpha}
\addcontentsline{toc}{chapter}{References}
\bibliography{references}

\end{document}